\documentclass{amsart}

\usepackage{amsmath,amsthm, amssymb}
\usepackage[all]{xy}
\usepackage{enumerate}
\usepackage{wasysym}
\usepackage{graphicx,color}

\newcommand{\R}{\mathbb R}

\renewcommand{\d}{\mathrm d}
\newcommand{\g}{\mathfrak{g}}

\newcommand{\Id}{I}

\newcommand{\ixi}{i_\xi}

\newcommand{\td}{\mathrm d}

\newcommand{\tprod}{\cup}

\newtheorem{theorem}{Theorem}[section]
\newtheorem{lemma}[theorem]{Lemma}

\theoremstyle{definition}

\theoremstyle{remark}

\numberwithin{equation}{section}
\title{Sasakian nilmanifolds}

\author[B. Cappelletti-Montano]{Beniamino Cappelletti-Montano}
 \address{Dipartimento di Matematica e Informatica, Universit\`a degli Studi di
 Cagliari, Via Ospedale 72, 09124 Cagliari, Italy}
 \email{b.cappellettimontano@gmail.com}

\author[A. De Nicola]{Antonio De Nicola}
 \address{CMUC, Department of Mathematics, University of Coimbra, 3001-501 Coimbra, Portugal}
 \email{antondenicola@gmail.com}

\author[J. C. Marrero]{Juan Carlos Marrero}
 \address{Unidad Asociada ULL-CSIC ``Geometr{\'\i}a Diferencial y Mec\'anica Geo\-m\'e\-tri\-ca''
Departamento de Matem\'aticas, Estad{\'\i}stica e Investigaci\'on Operativa, Facultad de Ciencias, Universidad de La Laguna, La Laguna, Tenerife, Spain}
 \email{jcmarrer@ull.edu.es}

\author[I. Yudin]{Ivan Yudin}
 \address{CMUC, Department of Mathematics, University of Coimbra, 3001-501 Coimbra, Portugal}
 \email{yudin@mat.uc.pt}


\thanks{Research partially supported by CMUC, funded by the European program
COMPETE/FEDER, by FCT (Portugal) grants PEst-C/MAT/UI0324/2011 (A.D.N. and I.Y.), by MICINN (Spain) grants
MTM2011-15725-E, MTM2012-34478 (A.D.N. and J.C.M.), the project of the Canary Government ProdID20100210 (J.C.M.), and by Prin 2010/11 -- Variet\`{a} reali e complesse: geometria, topologia e analisi armonica –- Italy (B.C.M.). \\ \medskip B.C.M., A.D.N. and I.Y. thank University of La Laguna for the hospitality during their staying.}

\begin{document}
\maketitle

\begin{abstract}
We prove that a compact nilmanifold admits a Sasakian structure if and only if it is a compact
quotient of the generalized Heisenberg group of odd dimension by a co-compact discrete subgroup.
\end{abstract}

\section{Introduction}

It is well known that the existence of a K\"ahler structure on a compact manifold $M$ of even dimension implies strong topological consequences
on $M$. In particular, any compact K\"ahler manifold satisfies the Hard Lefschetz property and is formal (see, for instance, \cite{DeGrMoSu,We}).

Using the second of these properties, a nice result which completely characterizes K\"ahler structures on compact nilmanifolds was found in  \cite{hasegawa}. The same result was obtained independently in  \cite{bensongordon}  using the Hard Lefschetz property. Namely,

\medskip

{\it A compact nilmanifold of even dimension admits a K\"ahler structure if and only it is diffeomorphic to a torus}.

\medskip

On the other hand, it is well known that the odd dimensional counterparts of K\"ahler manifolds are co-K\"ahler and Sasakian manifolds
(see \cite{blair2010,boyergalicki2008}). In some references (see \cite{blair2010,ChLeMa}), it is used the terminology cosymplectic manifolds for co-K\"ahler manifolds. However, in this note, we will use the last terminology which makes clear the close relation with K\"ahler manifolds. In fact, the term co-K\"ahler was used recently by Li \cite{Li} (see also \cite{BaLuOp,BaOp,survey}).
We remark that apart from the mathematical interest, co-K\"ahler and Sasakian manifolds are Poisson and contact manifolds, respectively, and that these last manifolds play an important role in some physical theories, particularly in time-dependent Mechanics (see \cite{AbMa,Arnold,CaLeMaMa,LeMaMa,LeMaMar}). Moreover, Sasakian manifolds have  recently attracted the interest of theoretical physicists, due to their role in the AdS/CFT duality that establishes a remarkable correspondence between gravity theories and gauge theories (see e.g. \cite{gabella2014, gauntlett2004, maldacena1998, martelli2006, martelli2008}). In addition, some interesting results on universal models for embeddings of compact Sasakian manifolds and on the global structure of these manifolds have been obtained recently (see \cite{OrVe1,OrVe2}; see also \cite{BoGaOr}).

Any compact co-K\"ahler manifold is formal \cite{ChLeMa}. So, using that a formal compact nilmanifold is diffeomorphic to a torus \cite{hasegawa}, we directly
deduce that

\medskip

{\it A compact nilmanifold of odd dimension admits a co-K\"ahler structure if and only if it is diffeomorphic to a torus}.

\medskip

This result is just the version for co-K\"ahler manifolds of the previous property for K\"ahler manifolds.

So, a natural question arise: what happens in the Sasakian setting with these results?

The aim of this paper is to give a complete answer to the previous question.

We remark that, very recently, a Hard Lefschetz theorem for Sasakian manifolds has been proved in \cite{CaNiYu}. However, so far, it is not clear if this result could be used in order to describe the compact Sasakian nilmanifolds.

On the other hand, differently from the K\"ahler case, compact Sasakian manifolds are not generally formal. Anyway, some interesting results have been obtained
very recently in this direction \cite{tievsky}. The geometric tool used in \cite{tievsky} is the basic cohomology with respect to the foliation
on the compact Sasakian manifold  $M$ which is generated by its Reeb vector field. In fact, in \cite{tievsky} the author proved that the real
homotopy type of a compact Sasakian manifold is a formal consequence of its basic cohomology and, in addition, its basic K\"ahler class.

Using this fact and some results in \cite{hasegawa} on minimal models of compact nilmanifolds, we give an answer to the previous question. More precisely, we prove the following result:

\begin{theorem}\label{thm:main}
A compact nilmanifold of dimension $2m+1$ admits a Sasakian structure if and only if it is a compact quotient of the generalized Heisenberg group $H(1, m)$ by a co-compact discrete subgroup $\Gamma$.
\end{theorem}

This is the main result of the paper.

We remark  that the generalized Heisenberg group $H(1, m)$ may be described as the group of real matrices
of the form
\[
\left(
\begin{array}{rcl}
1 & Q &    t \\
0 & I_m  & P \\
0  & 0  &   1
\end{array}
\right)
\]
with $Q = (q^1, \dots, q^m) \in \R^m$, $P^t = (p_1,\dots, p_m) \in \R^m$ and $t \in \R$. Its Lie algebra ${\frak h}(1, m)$ is isomorphic to a central extension of the abelian Lie algebra
of dimension $2m$ by a non-degenerate $2$-cocycle on it. Thus, one may choose a basis of ${\frak h}(1, m)$ in such a way
that the corresponding structure constants are rational numbers and, therefore, using a result in \cite{malcev}, we conclude that
$H(1, m)$ admits co-compact discrete subgroups (note that if $\Gamma(1, m)$ is the subgroup of matrices of $H(1, m)$ with
integer entries then $\Gamma(1, m)$ is a co-compact discrete subgroup).

On the other hand, our Theorem \ref{thm:main} extends some existing results in the literature. In particular, as a corollary
of Theorem 3.9 in \cite{afv2009}, one may deduce that a compact nilmanifold $G/\Gamma$ of dimension $2m+1$ admits a
Sasakian structure induced by a left-invariant Sasakian structure on $G$ if and only $G$ is isomorphic to $H(1, m)$.  We remark that
Theorem \ref{thm:main} takes care of the non-left-invariant Sasakian structures on nilmanifolds.

As we noted above, we do not know if the Hard Lefschetz theorem for Sasakian manifolds can be used for the proof of Theorem \ref{thm:main}. In fact, it is an open problem if there exist non-Heisenberg (and thus non-Sasakian) contact nilmanifolds that satisfy the Hard Lefschetz property (see \cite{CaNiYu}).

The paper is structured as follows. In Sections \ref{minimal-models-nilmanifolds} and \ref{Sasakian-manifolds}, we review some definitions and results on minimal models, compact nilmanifolds and Sasakian manifolds. Finally, in Section \ref{proof}, we prove Theorem \ref{thm:main}.

\section{Minimal models of nilmanifolds}\label{minimal-models-nilmanifolds}
In this section we summarize some definitions and results about Sullivan models of manifolds.

A \emph{(real) commutative differential graded algebra } $\left( A,\d \right)$
(CDGA for short) is a graded algebra $A = \bigoplus_{k\ge 0} A_k$ over $\mathbb{R}$ such that for all $x\in A_k$
and $y\in A_l$ we have
\[
x y = \left( -1 \right)^{kl} yx,
\]
together with a differential $\d$ of degree one, such that $\d(xy)=\d(x)y+(-1)^k x \d(y)$ and $\d^2 =0$.
An example of commutative differential graded algebra is given by the de Rham
complex $\left( \Omega^*\left( M \right), \d \right)$ of differential forms on a
smooth manifold $M$, with the multiplication given by the wedge product.

A morphism of CDGAs is a morphism of algebras which preserves the degree and commutes with the differentials.
For every CDGA $\left( A,\d \right)$ the cohomology algebra $H^*\left( A\right)$ can be considered as a CDGA with the zero differential.
Let $f:\left( A,\d \right) \longrightarrow \left( B ,\d \right)$ be a morphism of CDGAs. For every $k\geq 0$, the map $f$ induces a
morphism between the  $k$-th cohomologies
$$
H^k\left( f \right)\colon H^k\left( A \right) \to H^k\left( B \right).
$$
If all the morphisms $H^k\left( f \right)$ are isomorphisms then $f$ is called a \emph{quasi-isomorphism} of CDGAs.

A CDGA  $\left( A,\d \right)$ is said to be \emph{directly quasi-isomorphic} to  a CDGA $\left( B,\d \right)$ if there is a
quasi-isomorphism of CDGAs $f:\left( A,\d \right) \longrightarrow \left( B ,\d \right)$.
Two CDGAs $\left( A,\d \right)$ and $\left( B,\d \right)$ are \emph{quasi-isomorphic} if there is a chain of CDGAs $A=C_0$, $C_1$, \dots, $C_r = B$, such that either $C_j$ is directly
quasi-isomorphic to $C_{j+1}$ or $C_{j+1}$ is directly quasi-isomorphic to  $C_j$.

We say that  a  CDGAs $\left( A,\d \right)$ is connected if $H^{0}(A)=\mathbb{R}$. The reader can find the
definition of the minimal (Sullivan) algebra in \cite{felixopreatanre}. We will use the following facts on them. In every quasi-isomorphism class of connected CDGAs there is a unique (up to isomorphism) minimal algebra
$\left( \underline{A},\d \right)$. Moreover, for every CDGA in the class, there exists a quasi-isomorphism of CDGAs
$$f:\left(  \underline{A},\d \right) \longrightarrow \left( A ,\d \right).$$
The minimal algebra in the class of CDGAs quasi-isomorphic   to the given connected CDGA $\left( A ,\d \right)$ is called the  \emph{minimal model} of $\left( A ,\d \right)$.

We say that a CDGA $\left( A,\d \right)$ is a  \emph{model} for a manifold $M$ if $\left( A,\d \right)$ is quasi-isomorphic to $\left( \Omega^*\left( M \right),\d  \right)$. The minimal model of $\left( \Omega^*\left( M \right),\d  \right)$ will be also called the  \emph{minimal model} of $M$.

A \emph{nilmanifold} is a compact homogeneous space of a nilpotent Lie group.
 Malcev \cite{malcev} proved that any nilmanifold can be written as $G/\Gamma$, where $G$ is a simply-connected nilpotent Lie group and $\Gamma$ is a co-compact discrete subgroup.

We recall the following theorem of Hasegawa.
\begin{theorem}[\cite{hasegawa}] \label{thm:hasegawa}
The minimal model for a compact nilmanifold $G/\Gamma$ is given by
$(\wedge^*\g^*, \d)$, where $\g^*$ is the dual space of the Lie algebra $\g$ of the Lie group $G$ and $\d$ is the Chevalley-Eilenberg differential.
\end{theorem}

Suppose $\dim \g=2m+1$, with $m \geq 1$,  and let  $k-1\leq 2m+1$ be the dimension of  the first Chevalley-Eilenberg cohomology $H^1(\wedge^*\g^*)$ of $\g$. It is known that one can choose a basis $\{\alpha_1,\ldots,\alpha_{2m+1}\}$ of $\g^*$ such that $\alpha_1,\ldots,\alpha_{k-1}$ are closed, $\{[\alpha_1],\ldots,[\alpha_{k-1}]\}$ is a basis of $H^1(\wedge^*\g^*)$, and
\begin{equation}\label{eq:dal}
    d\alpha_l= -\sum_{i<j<l} \gamma_l^{ij} \alpha_i\wedge \alpha_j, \; \;  \mbox{for } 1\leq l\leq 2m+1, \mbox{ with } \gamma_l^{ij}=0, \mbox{ for } l<k.
\end{equation}

\section{Sasakian manifolds}\label{Sasakian-manifolds}
Let $M$ be a smooth manifold of dimension $2m+1$. A $1$-form $\eta$ on $M$ is called a
\emph{contact form} if $\eta\wedge (d\eta)^m$ nowhere vanishes.
Then the pair $(M,\eta)$ is called a (strict) \emph{contact manifold}.
We write $\Phi$ for $\frac12 d\eta$ and we denote by $\xi$ the  \emph{Reeb vector field},
that is the unique vector field on $M$ such that $i_{\xi}\eta=1$ and $i_{\xi}d\eta=0$.

Let $(M,\eta)$ be a contact manifold and $g$ a Riemannian metric on $M$.
 We
define the endomorphism $\phi\colon TM \to
TM$ by $ \Phi (X, Y ) = g( X, \phi Y)$.

Then $(M, \eta ,g)$ is called a \emph{Sasakian manifold} if the following conditions hold.
\begin{enumerate}[($i$)]
\item $\phi^2 =- \Id + \eta \otimes \xi$, where $\Id$ is the identity operator;
\item $g(\phi X , \phi Y )= g(X,Y) - \eta(X)\eta (Y)$ for any vector fields
	$X$ and $Y$ on $M$;
\item the normality condition is satisfied, namely
		\begin{equation*}
			 \left[ \phi, \phi \right]_{FN} +  2 d\eta \otimes
			\xi = 0,
		\end{equation*}
            where $[-,-]_{FN}$ is the Fr\"olicher-Nijenhuis
	    bracket.
\end{enumerate}

For a Sasakian manifold $(M, \eta, g)$, we will denote by $H^*_B(M)$ the basic cohomology of $M$ with respect to the foliation of dimension $1$ on $M$ which is generated by the Reeb vector field. It is clear that $d\eta$ is a basic $2$-form on $M$. Moreover, we will use the following result.
\begin{lemma}\label{em:detavol}
Let $M$ be a compact Sasakian manifold of dimension $2m+1$ with contact form $\eta$. Then $(\d\eta)^m = d\eta \wedge \ldots \wedge d\eta$ induces a non-zero element of the basic cohomology group $H^{2m}_B(M)$. Thus, $(d\eta)^{l}$ induces a non-zero element of $H^{2l}_{B}(M)$, for $1 \leq l \leq m$.
\end{lemma}
\begin{proof}
Since  $\d \eta$ is a basic form the same is true for all powers $(\d \eta)^{l}$. To prove the lemma it is enough to show that $[(\d\eta)^m]_{B}\neq 0$. Suppose that there exists a basic $(2m-1)$-form $\Omega$ on $M$ such that
$$(\d\eta)^m=\d\Omega.$$
Then
\begin{equation*}
     -\d(\eta\wedge\Omega)=-\d\eta\wedge\Omega+\eta\wedge\d\Omega=\eta\wedge(\d\eta)^m-\d\eta\wedge\Omega.
\end{equation*}
Now, $\d\eta\wedge\Omega$ is a $(2m+1)$-form on $M$ and
\begin{equation*}
     \ixi (\d\eta\wedge\Omega)=0.
\end{equation*}
This implies that $\d\eta\wedge\Omega=0$ and
\begin{equation*}
     \eta\wedge(\d\eta)^m=-\d(\eta\wedge\Omega).
\end{equation*}
Therefore, using Stoke's theorem
\begin{equation*}
     \int_M \eta\wedge(\d\eta)^m=-\int_M \d(\eta\wedge\Omega)=0,
\end{equation*}
which is a contradiction.
\end{proof}

For further details on Sasakian manifolds we refer the reader to \cite{blair2010} or \cite{boyergalicki2008}.

\section{Proof of Theorem \ref{thm:main}}\label{proof}

Let $H(1, m)$ be the generalized Heisenberg group of dimension $2m+1$. It is well known that $H(1, m)$ admits a left-invariant Sasakian structure $(\phi, \xi, \eta, g)$ (see, for instance, \cite{ChGo}). So, if $\Gamma$ is a co-compact discrete subgroup then $(\phi, \xi, \eta, g)$ induces a Sasakian structure on the compact nilmanifold $H(1, m)/\Gamma$.

Conversely, let $(M,\eta)$ be a contact manifold that admits a compatible Sasakian metric. Tievsky \cite{tievsky} proved that $M$ has the model
\begin{equation}\label{eq:tdga}
    (\mathcal{T}^*(M),\td),
\end{equation}
where
\begin{equation*}
    \mathcal{T}^*(M):=H_B^*(M)\otimes\R[y]/(y^2)
\end{equation*}
and we set $\deg(y)=1$.
The component of degree $p$ of Tievsky CDGA is given by
\begin{equation*}
    (H_B^*(M)\otimes\R[y]/(y^2))_p\cong H_B^p(M)\oplus H_B^{p-1}(M)y.
\end{equation*}
The differential $\td$ is defined by
\begin{equation}\label{eq:tievskyd}
    \td([\alpha]_B+[\beta]_B y):=[\beta\wedge \d \eta]_B\in H_B^{p+1}(M),
\end{equation}
where $\alpha$ is a basic closed $p$-form and $\beta$ is a basic closed $(p-1)$-form.

Now, let $(N,\eta)$ be a compact Sasakian nilmanifold of dimension $n=2m+1$, i.e. $N=G/\Gamma$ is a compact nilmanifold with a contact structure $\eta$ which admits a compatible Sasakian metric.

We will show that the Lie algebra $\g$ of $G$ is isomorphic to the Heisenberg Lie algebra ${\frak h}(1, m)$.

We have two models for $N$: the Tievsky model \eqref{eq:tdga} and the minimal model stated in Theorem \ref{thm:hasegawa}. Therefore there exists a quasi-isomorphism of CDGAs
\begin{equation}\label{eq:morphism}
f:(\wedge^*\g^*, \d)\longrightarrow (H_B^*(N)\otimes\R[y]/(y^2),\td).
\end{equation}

Note that
\begin{equation*}
    (H_B^*(N)\otimes\R[y]/(y^2))_1= H_B^1(N)\oplus \R y.
\end{equation*}
Define $k=\dim H^1(\wedge^*\g^*)+1$.
Let us choose a basis $\{\alpha_1,\ldots,\alpha_{2m+1}\}$ of $\g^*$ such that $\{[\alpha_1],\ldots,[\alpha_{k-1}]\}$ is a basis of $H^1(\wedge^*\g^*)$ and (\ref{eq:dal}) holds. We have for every $1\leq i \leq 2m+1$
\begin{equation}\label{eq:fai}
    f(\alpha_i)=\beta_i+ a_i y,
\end{equation}
for some $\beta_i\in H_B^1(N)$ and $a_i\in\R$.

From \eqref{eq:dal}, for every $1\leq i \leq 2m+1$ we get
\begin{equation}\label{eq:fdai}
   f(\d\alpha_i)=- \sum_{r<s<i} \gamma_i^{rs} f(\alpha_r)\tprod  f(\alpha_s),
\end{equation}
where we denoted by $\tprod $ the product in the Tievsky CDGA \eqref{eq:tdga}.
On the other hand, due to the definition \eqref{eq:tievskyd} of the Tievsky differential  we obtain from \eqref{eq:fai}
\begin{equation}\label{eq:dfai}
\td f(\alpha_i) = a_i[\d\eta]_B.
\end{equation}
Since $f$ is a morphism of CDGAs, we have $f(\d\alpha_i)=\td f(\alpha_i)$. Hence, from \eqref{eq:fdai} and \eqref{eq:dfai} it follows that
\begin{equation}\label{eq:detagamma}
   - \sum_{r<s<i} \gamma_i^{rs} f(\alpha_r)\tprod  f(\alpha_s)=f(\d\alpha_i)= a_i[\d\eta]_B  \qquad \mbox{for }1\leq i \leq 2m+1.
\end{equation}

Since, for $1\leq i \leq k-1$, we have $\d\alpha_i=0$, we get $a_i[\d\eta]_B=0$.
By Lemma \ref{em:detavol} $[\d\eta]_{B}\neq 0$, which immediately implies that $a_i=0$ for every $1\leq i \leq k-1$.
Therefore  we have
\begin{equation}\label{eq:fai2}
    f(\alpha_i)=\beta_i \qquad \mbox{for } 1\leq i \leq k-1.
\end{equation}

\begin{lemma}\label{lem:betas}
The set $\{\beta_1=f(\alpha_1),\ldots,\beta_{k-1}=f(\alpha_{k-1})\}$ is a basis of $H_B^1(N)$.
\end{lemma}
\begin{proof}
As first step we prove that
\begin{equation*}
    H^{1}(\mathcal{T}^*(N))=H^{1}_B(N).
\end{equation*}
Indeed, recall that $N=G/\Gamma$ and
\begin{equation*}
    \td_0:\mathcal{T}^0(N)=H^{0}(N)\longrightarrow \mathcal{T}^1(N)
\end{equation*}
is identically zero by \eqref{eq:tievskyd}. Therefore, from Lemma \ref{em:detavol}, we deduce that
\begin{align*}
    H^{1}(\mathcal{T}^*(N))=\ker \td_1= \{\beta + ay \,|\, \beta\in H^{1}_B(N) ,\, a\in\R,\,\td(\beta + ay)=0\}=H^{1}_B(N).
\end{align*}
Now, note that
\begin{equation*}
    H^1(f): H^1(\wedge^*\g^*) \longrightarrow H^{1}(\mathcal{T}^*(N))
\end{equation*}
is an isomorphism by assumption.
Since $\{\alpha_1=[\alpha_1],\ldots,\alpha_{k-1}=[\alpha_{k-1}]\}$ is a basis of $H^1(\wedge^*\g^*)$, we get that $\{\beta_1=f(\alpha_1),\ldots,\beta_{k-1}=f(\alpha_{k-1})\}$ is a basis of $H_B^1(N)$.
\end{proof}

Thus for $i\geq k$ we have
\begin{equation} \label{eq:betal}
     \beta_i= \sum_{r=1}^{k-1} s_{ir} \beta_r \qquad \mbox{for some } s_{ir}\in\R.
\end{equation}
Define
\begin{align*}
     \tilde{\alpha}_i&= \alpha_i \qquad &\mbox{for } 1\leq i \leq k-1,\\
     \tilde{\alpha}_i&= \alpha_i-\sum_{r=1}^{k-1} s_{ir} \alpha_r=\alpha_i-\sum_{r=1}^{k-1} s_{ir} \tilde{\alpha}_r  \qquad &\mbox{for } i \geq k.
\end{align*}
Then $\{ \tilde{\alpha}_1,\ldots,\tilde{\alpha}_{k-1},\ldots,\tilde{\alpha}_{2m+1}\}$ is a new basis of $\g^*$ such that
$\{ [\tilde{\alpha}_1],\ldots,[\tilde{\alpha}_{k-1}]\}$ is a basis of $H^1(\wedge^*\g^*)$, and for $l\geq k$ we have
\begin{align*}
    \d\tilde{\alpha}_l&=d\alpha_l -\sum_{r=1}^{k-1} s_{lr} \d\alpha_r= -\sum_{r<s<l} \gamma_l^{rs} \alpha_r\wedge \alpha_s\\
                      &=-\sum_{r<s<l} \gamma_l^{rs} \left(\tilde{\alpha}_r +\sum_{i=1}^{k-1} s_{ri} \tilde{\alpha}_i\right)\wedge \left(\tilde{\alpha}_s +\sum_{i=1}^{k-1} s_{si} \tilde{\alpha}_i \right)\\
                      &=-\sum_{r<s<l} \tilde{\gamma}_l^{rs} \tilde{\alpha}_r\wedge \tilde{\alpha}_s,
\end{align*}
for some new real numbers $\tilde{\gamma}_l^{rs}$.
Moreover, we get
\begin{equation}\label{eq:fatildel}
     f(\tilde{\alpha}_i)=a_i y \qquad \mbox{for every } i\geq k.
\end{equation}
Indeed, due to \eqref{eq:fai}, \eqref{eq:fai2} and \eqref{eq:betal}, we have
\begin{align*}
     f(\tilde{\alpha}_i)= f(\alpha_i)-\sum_{r=1}^{k-1} s_{ir} f(\alpha_r)
                        =\beta_i+a_iy-\sum_{r=1}^{k-1} s_{ir} \beta_r=a_i y.
\end{align*}

Now, we will prove that $a_i=0$ for every $i\leq 2m$.  Suppose this is not true. Then there is $l\leq 2m$ such that $a_l\neq 0$. It follows from Lemma \ref{em:detavol} that
\begin{equation*}
     (a_l[\d\eta]_B)^m=a_l^m[(\d\eta)^m]_B\neq 0.
\end{equation*}
But from \eqref{eq:detagamma} we get that
\begin{equation}\label{eq:fdal}
    (a_l[\d\eta]_B)^m= \left(-\sum_{r<s<l} \gamma_l^{rs} f(\alpha_r)\tprod  f(\alpha_s)\right)^m=0,
\end{equation}
since $l\leq 2m$ and thus in every product
\begin{equation*}
f(\alpha_{i_1})\tprod f(\alpha_{j_1}) \tprod \ldots \tprod  f(\alpha_{i_m})\tprod  f(\alpha_{j_m})=f(\alpha_{i_1}\wedge \alpha_{j_1}\wedge\dots\wedge\alpha_{i_m}\wedge \alpha_{j_m})
\end{equation*}
with $i_1<j_1<l,\ldots, i_m<j_m<l$ at least one index appears twice. Thus we have a contradiction. Therefore
\begin{equation}\label{eq:ai0}
a_i=0 \qquad \mbox{for every }i\leq 2m.
\end{equation}
Now, we will prove that $k=2m+1$.
Since $k=\dim H^{1}(\wedge^*\g^*)+1$,  we have $k\leq 2m+2$. By Nomizu theorem,
$H^{1}(\wedge^*\g^*)\cong H^{1}(N)$. Therefore $k-1=\dim H^{1}(N)=b_{1}(N)$. Since $N$ is Sasakian, $b_{1}(N)$ is even. Hence, we cannot have $k=2m+2$, and thus $k\leq 2m+1$.

Now, suppose that $k\leq 2m$. Then, from \eqref{eq:fatildel} and \eqref{eq:ai0} we get
\begin{equation}\label{eq:fatildei0}
     f(\tilde{\alpha}_i)=0 \qquad \mbox{for every } k\leq i\leq 2m.
\end{equation}
Therefore
\begin{equation}\label{eq:fatilde2m+1}
     f(\tilde{\alpha}_1 \wedge\ldots\wedge\tilde{\alpha}_{2m+1})=0.
\end{equation}
On the other hand, from \cite[Lemma 1]{hasegawa}
it follows that $\tilde{\alpha}_1 \wedge\ldots\wedge\tilde{\alpha}_{2m+1}$ is a generator of
$H^{2m+1}(\wedge^*\g^*)\cong \R$. Thus the cohomology class of $f(\tilde{\alpha}_i \wedge\ldots\wedge\tilde{\alpha}_{2m+1})$ should be a generator of $H^{2m+1}(\mathcal{T}^*(N))$, which contradicts to \eqref{eq:fatilde2m+1}. Therefore $k=2m+1$. This implies that
\[
f(\tilde{\alpha}_{2m+1})=a_{2m+1}y \quad \mbox{and }\quad f(\tilde{\alpha}_{j})=\beta_j,\quad  \mbox{for }\ j\leq 2m.
\]
As the cohomology class of
\[
a_{2m+1}\beta_1\tprod \ldots\tprod \beta_{2m}y=f(\tilde{\alpha}_1 \wedge\ldots\wedge\tilde{\alpha}_{2m+1})
\]
generates  $H^{2m+1}(\mathcal{T}^*(N))$, we conclude that $a_{2m+1}\neq 0$.

Thus, using Lemma \ref{em:detavol} and the fact that
$$f(\tilde{\alpha}_{2m+1} \wedge (\d \tilde{\alpha}_{2m+1})^{m}) = a_{2m+1}^{m+1} y \tprod [(d\eta)^m]_B,$$
we deduce that
$f(\tilde{\alpha}_{2m+1} \wedge (\d \tilde{\alpha}_{2m+1})^{m}) \neq 0$, which implies
\[
\tilde{\alpha}_{2m+1} \wedge (\d \tilde{\alpha}_{2m+1})^{m} \neq 0.
\]
In other words, $\tilde{\alpha}_{2m+1}\in {\frak g}^*$ is an algebraic contact structure on ${\frak g}$. Since
\begin{equation}\label{all-closed}
\d \tilde{\alpha}_i = 0, \; \; \; \mbox{ for } 1\leq i \leq 2m
\end{equation}
and $\d\tilde{\alpha}_{2m+1}$ is a linear combination of $\tilde{\alpha}_{i}\wedge \tilde{\alpha}_{j}$  with $i,j\leq 2m$ the Lie algebra ${\frak g}$ is $2$-step nilpotent. By Proposition 19 in \cite{GoRe}, we conclude
that ${\frak g}$ is the Heisenberg algebra ${\frak h}(1,m)$.

\end{document}